\documentclass[11pt,reqno]{amsart}
\usepackage{amsmath}
\usepackage{cases}
\usepackage{mathrsfs}
\usepackage{bbm}
\usepackage{amssymb}
\usepackage{amscd}
\usepackage{amsfonts,latexsym,amsmath,
amsthm,amsxtra,mathdots,amssymb,latexsym,mathabx}
\usepackage[all,cmtip]{xy}
\RequirePackage{amsmath} \RequirePackage{amssymb}
\usepackage{color}
\usepackage{colordvi}
\usepackage{multicol}
\usepackage{hyperref}
\usepackage{mathtools}
\usepackage[margin=1in]{geometry}
\usepackage{xcolor}

\hypersetup{
    colorlinks,
    linkcolor={red!50!black},
    citecolor={blue!100!black},
    urlcolor={blue!100!black}
}
\usepackage{cite}

\newcommand{\bea}{\begin{eqnarray}}
\newcommand{\eea}{\end{eqnarray}}
\newcommand{\bna}{\begin{eqnarray*}}
\newcommand{\ena}{\end{eqnarray*}}

\numberwithin{equation}{section}

\theoremstyle{plain}
\newtheorem{lemma}{Lemma}[section]
\newtheorem{theorem}[lemma]{Theorem}

\newtheorem{proposition}[lemma]{Proposition}

\theoremstyle{definition}

\newtheorem{remark}{Remark}

\renewcommand{\Re}{\operatorname{Re}}
\renewcommand{\Im}{\operatorname{Im}}

\renewcommand{\bmod}[1]{\,(\mathrm{mod\,}{#1})}
\renewcommand{\i}{\mathrm{i}}
\title[A Selberg-type zero-density result for twisted $\rm GL_2$  $L$-functions
and its application]
{A Selberg-type zero-density result for  twisted $\rm GL_2$ $L$-functions and its application}

\author{Qingfeng Sun}
	\address{School of Mathematics and Statistics, Shandong University, Weihai\\Weihai, 264209, China}
	
	\address{State Key Laboratory of Cryptography and Digital Economy
Security \\ Shandong University, Weihai, 264209, China}

\email{qfsun@sdu.edu.cn}

\author{Hui Wang}
  \address{School of Mathematics, Shandong University, Jinan 250100, China}
    \email{wh0315@mail.sdu.edu.cn}

\author{Yanxue Yu}
  \address{School of Mathematics and Statistics, Shandong University, Weihai\\Weihai, 264209, China}
    \email{yanxueyu@mail.sdu.edu.cn}



\thanks{Q. Sun was partially supported by the National Natural Science Foundation of China
(Grant Nos. 12471005 and 12031008) and
the Natural Science Foundation of Shandong Province (Grant No. ZR2023MA003).}

\date{}

\begin{document}
\begin{abstract}
Let $f$ be a fixed holomorphic primitive cusp form of even weight $k$,
level $r$ and trivial nebentypus $\chi_r$. Let $q$ be an odd prime with $(q,r)=1$
 and let $\chi$ be a primitive Dirichlet character modulus $q$ with $\chi\neq\chi_r$.
In this paper, we prove an
unconditional Selberg-type zero-density estimate for the family of twisted $L$-functions $L(s, f \otimes \chi)$ in the critical strip.
As an application, we establish an asymptotic formula for the even moments of
the argument function
$S(t, f \otimes \chi)=\pi^{-1}\arg L(1/2+\i t, f\otimes\chi)$ and
prove a central limit theorem for its distribution over $\chi$ of modulus $q$.

\end{abstract}
\maketitle

\section{Introduction}
\setcounter{equation}{0}

The zero-density estimates for automorphic $L$-functions play a pivotal role
in modern analytic number theory. These estimates provide a fundamental insight into the vertical distribution of
nontrivial zeros and are crucial in circumventing the Riemann Hypothesis
in numerous applications, such as bounding error terms in prime number theorems
or establishing equidistribution results.
To describe more precisely,
for $T>0$ and any fixed $\sigma\geq 0$, we define
\begin{equation*}
\begin{split}
N_{F}(\sigma; T)=\#\big\{\rho=1/2+\beta+\i\gamma: L(\rho,F)=0,\,
\sigma\leq\beta,\,|\gamma|\leq T\big\},
\end{split}
\end{equation*}
where $L(s,F)$ is an automorphic $L$-function associated with an automorphic form $F$,
which is absolutely convergent for $\Re(s)>1$,
admits an analytic continuation to the whole complex plane and satisfies a standard functional equation.

For the Riemann zeta function $\zeta(s)$, the classical zero-density estimate is due to Ingham \cite{Ingham}, who established the bound
\begin{equation*}\label{zero-density wrt zeta 1}
\begin{split}
N_{\zeta}(\sigma; T)\ll T^{3(1/2-\sigma)/(3/2-\sigma)}(\log T)^5,
\end{split}
\end{equation*}
uniformly for $0\leq \sigma \leq 1/2$.
This result was subsequently refined by several researchers, including
Montgomery, Huxley, Jutila, Heath-Brown, Ivi\'{c},
and more recently Guth and Maynard (see \cite{Montgomery, Huxley, Jutila, HB, Ivic, GM} and the references therein).

A major contribution to the theory was made by Selberg \cite[Theorem 1]{Selberg1},
who proved that for $0\leq \sigma \leq 1/2$,
\begin{equation*}\label{zero-density wrt zeta 2}
N_{\zeta}(\sigma; T)\ll T^{1-\sigma/4}\log T.
\end{equation*}
This estimate remains effective even when $\sigma$ is as small as $1/\log T$,
and it plays a crucial role in Selberg's unconditional proof that both the real
and imaginary parts of $\log \zeta(1/2 + \i t)$ are normally distributed over large intervals $t \in [T, 2T]$.
For Dirichlet $L$-functions $L(s,\chi)$,
Selberg \cite[Theorem 4]{Selberg2}
established an analogous zero-density estimate for
Dirichlet $L$-functions $L(s,\chi)$ over
the primitive character $\chi$ modulo a large integer $q$.
Using this estimate, he showed that for fixed $t$ the argument of
$L(1/2+\i t,\chi)$ becomes normally distributed as $\chi$ varies
modulo $q$, as $q\to\infty$.

While the zero-density theory has been extensively developed for
classical $L$-functions, such as
the Riemann zeta function and Dirichlet $L$-functions, and has achieved significant
progress for certain families of automorphic $L$-functions in higher-rank settings
(e.g., \cite{Luo1, KM, Hough, LS24, SW24}),
the zero-density behavior of twisted $L$-functions
remains poorly understood. In particular,
for the family of $L$-functions $L(s, f \otimes \chi)$,
where $f$ is a fixed modular form and $\chi$ varies over primitive
Dirichlet characters modulo $q$, there seems to be few results so far.

In this paper, we establish a Selberg-type zero-density estimate
for the twisted $L$-functions $L(s, f \otimes \chi)$ and apply it to study
the distribution of their argument functions.
Let $f$ be a fixed primitive holomorphic cusp form of even weight $k$, level $r$
and trivial nebentypus $\chi_r$
for the congruence subgroup $\Gamma_0(r)$,
where the Fourier expansion of $f$ at $\infty$ is given by
\bna
f(z)=\sum_{n=1}^\infty \lambda_f(n)n^{\frac{k-1}{2}}e(nz),
\ena
normalizd such that $\lambda_f(1)=1$.
The Ramanujan--Petersson conjecture,
which was proved by Deligne \cite{De}, asserts that
\bea\label{eq:KS}
|\lambda_{f}(n)|\leq\tau(n), \quad\text{for all}\,\, n \geq 1,
\eea
with $\tau(n)$ being the divisor function.
For $\Re(s)>1$, the Hecke $L$-function associated to $f$ is defined by
\begin{equation*}
\begin{split}
L(s,f)=\sum_{n=1}^\infty\frac{\lambda_f(n)}{n^s}
&=\prod_p\left(1-\frac{\lambda_f(p)}{p^{s}}
+\frac{\chi_r(p)}{p^{2s}}\right)^{-1}\\
&=\prod_p\left(1-\frac{\alpha_f(p)}{p^{s}}\right)^{-1}
\left(1-\frac{\beta_f(p)}{p^{s}}\right)^{-1},
\end{split}
\end{equation*}
where $\alpha_f(p)$ and $\beta_f(p)$ are local roots of $L(s,f)$ at $p$, satisfying
\bea\label{Deligne bound}
\alpha_f(p)+\beta_f(p)=\lambda_f(p),\quad
\alpha_f(p)\beta_f(p)=\chi_r(p)
\quad\text{and}\quad|\alpha_f(p)|,\,|\beta_f(p)|\leq1.
\eea

Let $q$ be an odd prime, $(q,r)=1$, $\varphi^*(q)=q-2$ denote the number of primitive characters modulo $q$.
Let $\chi$ be a primitive Dirichlet character modulus $q$ with $\chi\neq\chi_r$.
Note that $f\otimes\chi$ is a primitive cusp form for the group $\Gamma_0(q^2r)$ with
central character $\chi^2\chi_r$.
The twisted $L$-function associated with it is defined by
\begin{equation*}
\begin{split}
L(s, f\otimes\chi)
=\sum_{n=1}^{\infty}\frac{\lambda_f(n)\chi(n)}{n^s}
&=\prod_p\left(1-\frac{\lambda_f(p)\chi(p)}{p^{s}}
+\frac{\chi_r(p)\chi^2(p)}{p^{2s}}\right)^{-1}\\
&=\prod_p\left(1-\frac{\alpha_f(p)\chi(p)}{p^{s}}\right)^{-1}
\left(1-\frac{\beta_f(p)\chi(p)}{p^{s}}\right)^{-1},
\quad\Re(s)>1.
\end{split}
\end{equation*}
It is known that $L(s,f\otimes\chi)$ admits an analytic continuation to
the whole complex plane $\mathbb{C}$ as an entire function
and satisfies the functional equation (see Iwaniec--Kowalski \cite[Proposition 14.20]{IK})
\bna
\Lambda(s,f\otimes\chi)
=\varepsilon(f\otimes\chi)
\Lambda(1-s,f\otimes\overline{\chi}),
\ena
with
\bna
\Lambda(s,f\otimes\chi)
=\left(\frac{q\sqrt{r}}{2\pi}\right)^s
\Gamma\left(s+\frac{k-1}{2}\right)
L(s,f\otimes\chi)
\ena
and
\bna
\varepsilon(f\otimes\chi)
=\varepsilon(f)\chi(r)\varepsilon_\chi^2,
\ena
where $\varepsilon(f)$ is the root number of
$L(s,f)$ satisfying $|\varepsilon(f)|=1$ and
$\varepsilon_\chi$ is the normalized Gauss sum.

Our first main result is the following zero-density estimate for twisted $L$-functions $L(s, f\otimes \chi)$.
\begin{theorem}\label{zero-density}
Let $q$ be a prime number. For $\sigma>1/2$,
we let
\bna
N(f\otimes\chi;\sigma,t_1,t_2)
:=\#\{\rho=\beta+\mathrm{i}\gamma \,|\, L(\rho,f\otimes\chi)=0,
\, \beta\geq\sigma,\; t_1\leq\gamma\leq t_2\}
\ena
and
\bea\label{N(sigma,H)}
N(f;\sigma, t_1,t_2):=\frac{1}{\varphi^*(q)}\;\;\sideset{}{^*}\sum_{\chi\bmod q}N(f\otimes\chi;\sigma,t_1,t_2).
\eea
There exists an absolute constant $A>0$ such that
for any real numbers $t_1$ and $t_2$ with
\begin{equation*}
\begin{split}
t_1<t_2\quad\text{and}\quad t_2-t_1\geq1/\log q,
\end{split}\end{equation*}
for any $1/2+1/\log q\leq\sigma\leq 1$ and for any $c$ with $0<c<1/360$, one has
\bna
N(f;\sigma, t_1,t_2)
\ll(1+|t_1|+|t_2|)^A(t_2-t_1)
q^{-\eta(\sigma-1/2)}\log q,
\ena
where $\eta=\eta(c)>0$ is a constant and the implied constant depends only on $c$.
\end{theorem}

The zero-density estimates serve not only as a tool for bounding the number of zeros off the critical line but also play a crucial role in understanding the statistical behavior of the argument functions of $L$-functions. One of the most prominent examples is the study of the function
\bna
S(t)=\frac{1}{\pi}\arg \zeta(1/2+\i t),
\ena
which is intimately connected to the distribution of the nontrivial zeros of the Riemann zeta function $\zeta(s)$ up to height $t$.
Defined via continuous variation along the horizontal line
from $\infty + \i t$ to $1/2 + \i t$, starting at $\infty + \i t$ with initial value 0,
the behavior of $S(t)$ has been studied extensively
by numerous number theorists (see \cite{Backlund, Cramer, Littlewood, Titchmarsh1}).

Selberg \cite{Selberg, Selberg1} initiated the study of the distribution of $S(t)$ and proved that for any integer $n\geq1$,
\bna
\frac{1}{T}\int_{T}^{2T}S(t)^{2n}\mathrm{d}t=\frac{(2n)!}{n!(2\pi)^{2n}}(\log\log T)^n+O\big((\log\log T)^{n-1/2}\big),
\ena
which indicates that $|S(t)|$ typically has size $\sqrt{\log\log T}$ in an average sense.
Subsequently, Selberg \cite{Selberg2} generalized this investigation to the family of Dirichlet $L$-functions. For a fixed $t>0$, he showed that
\bna
\frac{1}{q}\,\,\sideset{}{^\star}\sum_{\chi \bmod q} S(t, \chi)^{2n}=\frac{(2n)!}{n!(2\pi)^{2n}}(\log\log q)^n+O\big((\log\log q)^{n-1/2}\big),
\ena
where $S(t, \chi) = \frac{1}{\pi} \arg L(1/2 + \i t, \chi)$ and the summation is over primitive characters modulo $q$.
Analogous results for odd moments can also be derived with slight modifications of Selberg's method.

Parallel results in the $\rm GL_2$ and $\rm GL_3$ settings have been established in various contexts.
For $\rm GL_2$ holomorphic cusp forms, unconditional results were obtained by Liu--Shim~\cite{LS}
in the weight aspect and by the first two authors~\cite{SW} in the level aspect, respectively.
The case of $\rm GL_2$ Hecke--Maass forms, analogous results were initially proved
under the assumption
of the Generalized Riemann Hypothesis (GRH) by Hejhal--Luo~\cite{HL}, and later rendered
unconditionally by the first two authors~\cite{SW2} through the application of weighted zero-density estimates.
More recently, the first two authors~\cite{SW25} established
an unconditional asymptotic formula for the moments of
$
S_F(t) = \pi^{-1}\arg L(1/2 + \i t, F),
$
where $F$ is a Hecke--Maass cusp form for $\mathrm{GL}_3$ in the generic position.
This work removes the reliance on GRH required in the earlier work of Liu--Liu~\cite{LL}.

Inspired by these developments, we aim to investigate the behavior of the argument function
\[
S(t, f\otimes\chi) := \frac{1}{\pi}\arg L(1/2 + \i t, f\otimes \chi),
\]
averaged over primitive characters modulo $q$.
Using Theorem~\ref{zero-density} as a key input,
we prove an unconditional asymptotic formula for the even moments of $S(t, f \otimes \chi)$,
and derive a quantitative central limit theorem.

\begin{theorem}\label{main-theorem}
Let $t>0$ and $n\in \mathbb{N}$ be given.
For sufficiently large prime number $q$, we have
\bea\label{main result}
\frac{1}{\varphi^*(q)}\;\;\sideset{}{^*}\sum_{\chi\bmod q}S(t,f\otimes\chi)^{2n}
=\frac{(2n)!}{n!(2\pi)^{2n}}
\left(\log\log q\right)^{n}
+O\left(\left(\log\log q\right)^{n-\frac{1}{2}}\right),
\eea
where $*$ denotes the summation goes through the primitive characters $\chi \bmod q$.

Moreover, we define the following probability measure
$\mu_{q}$ on $\mathbb{R}$ by
\bna
\mu_q(E):=\frac{1}{\varphi^*(q)}\;\;\sideset{}{^*}\sum_{\chi\bmod q}\textbf{1}_E\left(\frac{S(t,f\otimes\chi)}{\sqrt{\log\log q}}\right),
\ena
where $\mathbf{1}_E$ is the characteristic function of a Borel measurable set $E$ in $\mathbb{R}$.
Then for every $a<b$, we have
\bna
\lim_{q\rightarrow \infty} \mu_q([a,b])=\sqrt{\pi}\int_a^b \exp(-\pi^2\xi^2)\mathrm{d}\xi,
\ena
uniformly with respect to $q$.
\end{theorem}
\begin{remark}
For a holomorphic primitive cusp form  $f$
with even weight $k$, level $r$ and trivial nebentypus $\chi_r$,
a primitive Dirichlet character $\chi$ modulus prime $q$ with $(q,r)=1$ and $\chi\neq\chi_r$,
our result indicates that $|S(t,f\otimes\chi)|$ on average has order of magnitude $\sqrt{\log\log q}$. In addition, as $q\rightarrow \infty$, the ratios $S(t,f\otimes\chi)/\sqrt{\log\log q}$ converges in distribution
to a normal distribution of mean $0$
and variance $(2\pi^2)^{-1}$.
\end{remark}

\medskip

\noindent\textbf{Notation}
Throughout the paper,
$\varepsilon$ is an arbitrarily small positive
real number and $A$ is a sufficiently large positive number which may not be the same at each occurrence.
As usual, $e(z) = \exp (2 \pi \i z) = e^{2 \pi \i z}$.
We write $f = O(g)$ or $f \ll g$ to mean $|f| \leq Cg$ for some unspecified positive constant $C$.

\medskip

\noindent\textbf{Structure of this paper.} This paper is organized as follows.
In Section~\ref{prelim}, we provide essential preliminary results needed for the proof.
In Section~\ref{Prove Theorem 1.1}, we present the detailed proofs of Theorems \ref{zero-density} and \ref{main-theorem}.
In Section~\ref{approximation of S(t)}, we establish some lemmas and provide a crucial approximation of $ S(t,f\otimes\chi)$.
Finally, in Section~\ref{Proof of the main theorem},
we give the detailed proof of Proposition~\ref{main propsition}.

\section{Preliminaries}\label{prelim}
Firstly, we collect some classical results.
\begin{lemma}\label{bound} Let $x\geq 2$. The following statements hold.

(1) There exist constants $c_1$ and $c_2$, such that
\bna\label{classical bound 1}
\sum_{p\leq x}\frac{1}{p}=\log\log x+c_1+O(e^{c_2\sqrt{\log x}}).
\ena

(2) We have
\bna\label{classical bound 3}
\sum_{p\leq x}\frac{\log p}{p}=\log x+O(1),
\ena
and
\bna\label{classical bound 2}
\sum_{p\leq x}\frac{(\log p)^2}{p}\ll (\log x)^2.
\ena

(3) Let $f$ be a cuspidal holomorphic Hecke eigenforms of even weight $k$ for the modular group $\rm SL_2(\mathbb{Z})$ with normalized Fourier coefficients $\lambda_f(n)$, then we have
\bna\label{selberg orthogonal}
\sum_{p\leq x}\frac{|\lambda_f(p)|^2}{p}=\log\log x+O(1).
\ena
\end{lemma}
\begin{proof}
The first and second assertions come from Davenport \cite[pp.\,56--57]{Dav},
and the third assertion is from Liu--Wang--Ye \cite[Corollary 1.5]{LWY}.
\end{proof}
\begin{lemma}\label{lem18}
Let $1<y\leq q^{1/n}$. If $|a_{p}|<A\frac{\log p}{\log y}$ for all $p\leq y$, we have
\bna
\sideset{}{^*}\sum_{\chi\bmod q}\left| \sum_{p < y} \frac{a_p}{\sqrt{p}} \chi(p) \right|^{2n}=O(q),
\ena
and if $|a_p'|<A$ for $p\leq y$,
\bna
\sideset{}{^*}\sum_{\chi\bmod q}\left| \sum_{p < y} \frac{a_p'}{p} \chi(p^2) \right|^{2n}=O(q).
\ena
\end{lemma}
\begin{proof}
See Selberg \cite[Lemma 18]{Selberg2}.
\end{proof}

We use the following version of the argument principle
and the mollified second moment
to get the zero-density estimate 
for the twisted $L$-functions.
\begin{lemma}\label{The argument principle}
Let $\omega$ be a holomorphic function in a neighborhood of the domain $\Re(s) \geq \sigma'$, $t_1 \leq\Im(s) \leq t_2$. Assume that in this region it satisfies
\bna
\omega(s)=1+o\left(\exp\Big(-\frac{\pi}{t_2-t_1}\Re(s)\Big)\right),
\ena
uniformly as $\Re(s)\rightarrow+\infty$.
Denote the zeros of $\omega$ (in the interior of this region) by $\rho=\beta+\mathrm{i} \gamma$. Then we have
\begin{equation*}
\begin{split}
2(t_2-t_1)\sum_{\beta\geq\sigma'\atop t_1\leq\gamma\leq t_2}
&\sin\left(\pi\frac{\gamma-t_1}{t_2-t_1}\right)
\sinh\left(\pi\frac{\beta-\sigma'}{t_2-t_1}\right)\\
&=\int_{t_1}^{t_2}\sin\left(\pi\frac{t-t_1}{t_2-t_1}\right)
\log |\omega(\sigma'+\mathrm{i} t)|\mathrm{d}t\\
&\quad+\int_{\sigma'}^{+\infty}\sinh\left(\pi\frac{\beta-\sigma'}{t_2-t_1}\right)
(\log|\omega(\beta+\mathrm{i}  t_1)|+\log|\omega(\beta+\mathrm{i}  t_2)|)\mathrm{d}\beta,
\end{split}
\end{equation*}
where the zeros are counted according to multiplicity.
\end{lemma}
\begin{proof}
See Selberg \cite[Lemma 14]{Selberg2}.
\end{proof}

\begin{lemma}\label{second moment}
For any fixed $c$ with $0 < c< 1/360$,
define $L=q^c$ and the mollifier
\begin{equation*}
\begin{split}
M(s,f\otimes\chi)=\sum_{\ell\leq L}x_\ell\frac{\chi(\ell)}{\ell^{s}},
\end{split}
\end{equation*}
where
\begin{equation*}
\begin{split}
x_\ell=\mu_f(\ell)P\left(\frac{\log(L/\ell)}{\log L}\right)\delta_{\ell\leq L}\delta_{(\ell, r)=1}.
\end{split}
\end{equation*}
Here $\mu_f(n)$ denotes the convolution inverse of the Hecke eigenvalues $\lambda_f(n)$ and
\begin{equation*}
\begin{split}
P(x)=\begin{cases}
2x, \,&\text{if}\,\,\,0\leq x\leq 1/2,\\
1, \,&\text{if}\,\,\,1/2\leq x\leq 1.
\end{cases}
\end{split}
\end{equation*}
Then there exist constants $A > 0$ and $\eta=\eta(c)>0$ such that for any prime $q\geq2$, we have
\bna
\frac{1}{\varphi^*(q)}\;\;\sideset{}{^*}\sum_{\chi\bmod q}\;
\Big|L(\sigma+\i t,f\otimes\chi)M(\sigma+\i t,f\otimes\chi)-1\Big|^2
\ll(1+|t|)^Aq^{-\eta(\sigma-1/2)}
\ena
holds uniformly for $\sigma\geq1/2-1/\log q$,
where the implied constant depends only on $c$.
\end{lemma}
\begin{proof}
See Blomer \emph{et al.} \cite[Theorem 8.5]{BFKMMS}.
\end{proof}
\begin{remark}
In view of this lemma, for $s=\sigma+\i t$ with $\sigma\geq1/2-1/\log q$, we have
\begin{equation*}
\begin{split}
LM(s,f\otimes\chi):=L(s,f\otimes\chi)M(s,f\otimes\chi)=1+
O((1+|t|)^{A/2}q^{-\eta(\sigma-1/2)/2})
\end{split}
\end{equation*}
for each primitive Dirichlet character $\chi$ modulus $q$.
Thus $LM(s,f\otimes\chi)$ is non-vanishing for sufficiently large $\sigma$.
\end{remark}

Finally, we need the lemma of the theory of moments in probability theory (see, for example, Billingsley \cite[Theorem 30.2]{Billingsley}).
\begin{lemma}\label{probability theory}
Suppose that the distribution of random variable $X$ is determined by its moments, that the $X_n$ have moments of all orders, and that
$\lim_n E[X_n^r]=E[X^r]$ for $r=1,2,\cdots$. Then $X_n\Rightarrow X$.
\end{lemma}

\section{Proof of Theorems \ref{zero-density} and \ref{main-theorem}}
\label{Prove Theorem 1.1}
\subsection{Proof of Theorem \ref{zero-density}}
In this subsection, we give the detailed proof of the zero-density estimates,
following Selberg \cite[Theorem 4]{Selberg2} and Kowalski \cite[Theorem 14]{K}.
It is sufficient to consider the case that $t_2-t_1=1/\log q$.
Let $L=q^c$ with $0<c<1/360$, and write
\bna
t_1'=t_1-\frac{\varpi}{\log q},\quad
t_2'=t_2+\frac{\varpi}{\log q},\quad
\sigma'=\sigma-\frac{1}{2\log q},
\ena
where $\varpi$ is a positive real number such that
\bna
\eta>\frac{\pi}{2\varpi+1},
\ena
with \( \eta = \eta(c) \) as in Lemma~\ref{second moment}.
For any $1/2+1/\log q\leq\sigma\leq 1$, we have
\begin{equation*}\begin{split}
N(f\otimes\chi;\sigma,t_1,t_2)
&:=\sum_{L(f\otimes\chi,\beta+\i \gamma)=0
\atop\beta\geq\sigma,\; t_1\leq\gamma\leq t_2}1
\leq2\log q\sum_{L(f\otimes\chi,\beta+\i \gamma)=0
\atop\beta\geq\sigma',\; t_1\leq\gamma\leq t_2}(\beta-\sigma')\\
&\leq 2(t_2'-t_1')\log q \sum_{L(f\otimes\chi,\beta+\i \gamma)=0
\atop\beta\geq\sigma',\; t_1'\leq\gamma\leq t_2'}
\sin\left(\pi\frac{\gamma-t_1'}{t_2'-t_1'}\right)
\sinh\left(\pi\frac{\beta-\sigma'}{t_2'-t_1'}\right).
\end{split}\end{equation*}
Since the zeros of $L(s,f\otimes\chi)$ are also zeros of
the function
\bna
1-(LM(s,f\otimes\chi)-1)^2,
\ena
we can apply Lemma \ref{The argument principle}
with
\bna
\omega(s)=1-(LM(s,f\otimes\chi)-1)^2,
\ena
and obtain that $N(f\otimes\chi;\sigma,t_1,t_2)$ is
\begin{equation*}\begin{split}
&\leq(\log q)\int_{t_1'}^{t_2'}
\sin\left(\pi\frac{t-t_1'}{t_2'-t_1'}\right)
\log |\omega(\sigma'+\i t)|\mathrm{d}t\\
&\quad+(\log q)\int_{\sigma'}^{+\infty}
\sinh\left(\pi\frac{\beta-\sigma'}{t_2'-t_1'}\right)
(\log|\omega(\beta+\i t_1')|+\log|\omega(\beta+\i t_2')|)\mathrm{d}\beta.
\end{split}\end{equation*}
Since
\bna
\log|1+x|\leq\log(1+|x|)\leq |x|, \quad\text{for}\,\,x\neq-1,
\ena
and by the definition \eqref{N(sigma,H)} and Lemma \ref{second moment}, $N(f;\sigma, t_1,t_2)$ is
\begin{equation*}\begin{split}
\leq&\,(\log q)\int_{t_1'}^{t_2'}\sin\left(\pi\frac{t-t_1'}{t_2'-t_1'}\right)
\frac{1}{\varphi^*(q)}\;\;\sideset{}{^*}\sum_{\chi\bmod q}
|LM(\sigma'+\i t,f\otimes\chi)-1|^2\mathrm{d}t\\
&+(\log q)\sum_{j=1}^2\int_{\sigma'}^{\infty}
\sinh\left(\pi\frac{\beta-\sigma'}{t_2'-t_1'}\right)
\frac{1}{\varphi^*(q)}\;\;\sideset{}{^*}\sum_{\chi\bmod q}
|LM(\beta+\i t_j',f\otimes\chi)-1|^2\mathrm{d}\beta\\
\ll&\,(1+|t_1|+|t_2|)^A(t_2-t_1)q^{-\eta(\sigma-1/2)}
\log q.
\end{split}\end{equation*}
This completes the proof of Theorem \ref{zero-density}.

\subsection{Proof of Theorem \ref{main-theorem}.}
For a positive parameter $x$ (to be chosen later), let
\bea\label{MR}
M(t,f\otimes\chi):=\frac{1}{\pi}\Im\sum_{p\leq x^3}\frac{C_f(p)\chi(p)}{p^{1/2+\i t}},\quad
\text{and}\quad
R(t,f\otimes\chi):=S(t,f\otimes\chi)-M(t,f\otimes\chi),
\eea
where we recall from \eqref{Deligne bound} and \eqref{Log derivative coefficient}
that $C_f(p)=\alpha_f(p)+\beta_f(p)=\lambda_f(p)$.
We now state the following key proposition,
whose detailed proof will be presented in Section~\ref{Proof of the main theorem}.
\begin{proposition}\label{main propsition}
Let $t>0$ be given. For prime number $q$ sufficiently large and $x=q^{\delta/3}$ with sufficiently small $\delta>0$, we have
\bea\label{M(t,f) moment}
\frac{1}{\varphi^*(q)}\;\;\sideset{}{^*}\sum_{\chi \bmod q}M(t,f\otimes\chi)^{2n}
=\frac{(2n)!}{n!(2\pi)^{2n}}\left(\log\log q\right)^{n}
+O\left(\left(\log\log q\right)^{n-1}\right),
\eea
and
\bea\label{R(t,f) moment}
\frac{1}{\varphi^*(q)}\;\;\sideset{}{^*}\sum_{\chi \bmod q}|R(t,f\otimes\chi)|^{2n}
=O_{t,n}(1).
\eea
\end{proposition}
Now we begin to give the proof of Theorem \ref{main-theorem}.
By the binomial theorem, we write
\begin{equation}
\begin{split}\label{after binomia tm}
S(t,f\otimes\chi)^{2n}
=M(t,f\otimes\chi)^{2n}
+O_n\left(\sum_{\ell=1}^{2n}|M(t,f\otimes\chi)|^{2n-\ell}
|R(t,f\otimes\chi)|^{\ell}\right),
\end{split}
\end{equation}
where $M(t,f\otimes\chi)$ and $R(t,f\otimes\chi)$ are defined in \eqref{MR}.
Then we have
\begin{equation*}
\begin{split}
\frac{1}{\varphi^*(q)}\;\;\sideset{}{^*}\sum_{\chi(\text{{\rm mod }} q)}S(t,f\otimes\chi)^{2n}
&=\frac{1}{\varphi^*(q)}\;\;\sideset{}{^*}\sum_{\chi(\text{{\rm mod }} q)}M(t,f\otimes\chi)^{2n}\\
&\quad+O_n\left(\sum_{\ell=1}^{2n}\;\;\frac{1}{\varphi^*(q)}\;\;\sideset{}{^*}\sum_{\chi(\text{{\rm mod }} q)}
|M(t,f\otimes\chi)|^{2n-\ell}|R(t,f\otimes\chi)|^{\ell}\right).
\end{split}
\end{equation*}
For $1\leq\ell\leq 2n$, we apply the generalized H\"{o}lder's inequality and Proposition \ref{main propsition}, and get
\begin{equation*}
\begin{split}
&\frac{1}{\varphi^*(q)}\;\;\sideset{}{^*}\sum_{\chi(\text{{\rm mod }} q)}
|M(t,f\otimes\chi)|^{2n-\ell}|R(t,f\otimes\chi)|^{\ell}\\
&\ll\left(\frac{1}{\varphi^*(q)}\;\;\sideset{}{^*}\sum_{\chi(\text{{\rm mod }} q)}|M(t,f\otimes\chi)|^{2n}\right)^{\frac{2n-\ell}{2n}}
\left(\frac{1}{\varphi^*(q)}\;\;\sideset{}{^*}\sum_{\chi(\text{{\rm mod }} q)}
|R(t,f\otimes\chi)|^{2n}\right)^{\frac{\ell}{2n}}\\
&\ll_{t,n}\left(\log\log q\right)^{n-\frac{\ell}{2}}
\ll_{t,n}\left(\log\log q\right)^{n-\frac{1}{2}}.
\end{split}
\end{equation*}
Then the asymptotic formula \eqref{main result} follows from \eqref{M(t,f) moment}, \eqref{after binomia tm}
and the above estimate.

For the second assertion of Theorem \ref{main-theorem},
note that the $2n$-th moment of $\mu_q$ is
\bna
\int_{\mathbb{R}}\xi^{2n}\mathrm{d}\mu_q(\xi)
=\frac{1}{\varphi^*(q)}\;\;\sideset{}{^*}\sum_{\chi\bmod q}\left(\frac{S(t,f\otimes\chi)}{\sqrt{\log\log q}}\right)^{2n}.
\ena
Using the asymptotic formula \eqref{main result}, we obtain that for all $n$,
\bna
\lim_{q\rightarrow \infty}
\int_{\mathbb{R}}\xi^{2n}\mathrm{d}\mu_q(\xi)
=\frac{(2n)!}{n!(2\pi)^{2n}}
=\sqrt{\pi}\int_{\mathbb{R}}\xi^{2n} \exp(-\pi^2\xi^2)\mathrm{d}\xi.
\ena
By the theory of moments in probability theory (see Lemma \ref{probability theory}),
we complete the proof of Theorem \ref{main-theorem}.

\section{An approximation for $S(t,f\otimes\chi)$}\label{approximation of S(t)}
In this section we will prove several lemmas and derive an approximation for $S(t,f\otimes\chi)$.
We denote by $\rho=\beta+\i \gamma$ a typical zero of $L(s,f\otimes\chi)$ inside the critical strip, i.e., $0<\beta<1$.
For $\Re(s)>1$, we have the Dirichlet series expansion for the logarithmic derivative of $L(s,f\otimes\chi)$,
\bea\label{Log derivative definition}
-\frac{L'}{L}(s,f\otimes\chi)
=\sum_{n=1}^\infty \frac{\Lambda(n)C_f(n)\chi(n)}{n^s},
\eea
where $\Lambda(n)$ denotes the von Mangoldt function, and
\bea\label{Log derivative coefficient}
C_f(n)=\begin{cases}
\alpha_f(p)^m+\beta_f(p)^m, \,&\text{if}\,\,n=p^m\,\,\text{for a prime}\,\,p,\\
0, \,& \text{otherwise}.
\end{cases}
\eea

\begin{lemma}\label{lemma 1}
Let $x>1$. For $s\neq \rho$ and $s\neq -m-\frac{k-1}{2}$ with $m\geq 0$, we have the following identity
\begin{equation*}
\begin{split}
\frac{L'}{L}(s,f\otimes\chi)
=&-\sum_{n\leq x^3} \frac{C_f(n)\Lambda_x(n)\chi(n)}{n^s}
-\frac{1}{\log^2x}\sum_\rho\frac{x^{\rho-s}(1-x^{\rho-s})^2}{(\rho-s)^3}\\
&-\frac{1}{\log^2x}\sum_{m=0}^\infty
\frac{x^{-m-\frac{k-1}{2}-s}(1-x^{-m-\frac{k-1}{2}-s})^2}
{(-m-\frac{k-1}{2}-s)^3},
\end{split}
\end{equation*}
where
\begin{equation*}
\begin{split}
\Lambda_x(n)=\begin{cases}
\Lambda(n), \,&\textit{if }\,\,n\leq x,\\
\Lambda(n)\frac{\log^2(x^3/n)-2\log^2(x^2/n)}{2\log^2x}, \,& \textit{if }\,\,x\leq n\leq x^2,\\
\Lambda(n)\frac{\log^2(x^3/n)}{2\log^2x}, \,& \textit{if }\,\,x^2\leq n\leq x^3,\\
0, \,& \textit{if }\,\,n\geq x^3.\\
\end{cases}
\end{split}
\end{equation*}
\end{lemma}
\begin{proof}
Recall the discontinuous integral
\bea\label{discontinuous integral}
\frac{1}{2\pi \i}\int_{(\alpha)}\frac{y^s}{s^3}\mathrm{d}s=
\begin{cases}
\frac{\log^2 y}{2}, \,&\textit{if }\,\, y\geq 1,\\
0, \,& \textit{if }\,\, 0<y\leq 1
\end{cases}
\eea
for $\alpha>0$.
It follows from \eqref{Log derivative definition} and \eqref{discontinuous integral} that
\bna
-\log^2x\sum_{n=1}^\infty \frac{C_f(n)\Lambda_x(n)\chi(n)}{n^s}
=\frac{1}{2\pi \i}\int_{(\alpha)}\frac{x^u(1-x^u)^2}{u^3}
\frac{L'}{L}(s+u,f\otimes\chi)\mathrm{d}u,
\ena
where $\alpha=\max \{2, 1+\Re(s)\}$.
By shifting the line of integration all way to the left, we pick up the residues at $u=0$, $u=\rho-s$ and $u=-m-\frac{k-1}{2}-s$ ($m= 0,1,2, \cdots$) and deduce that
\begin{equation*}
\begin{split}
&\frac{1}{2\pi \i}\int_{(\alpha)}\frac{x^u(1-x^u)^2}{u^3}
\frac{L'}{L}(s+u,f\otimes\chi)\mathrm{d}u\\
&=\frac{L'}{L}(s,f\otimes\chi)\log^2x
+\sum_\rho\frac{x^{\rho-s}(1-x^{\rho-s})^2}{(\rho-s)^3}
+\sum_{m=0}^\infty
\frac{x^{-m-\frac{k-1}{2}-s}(1-x^{-m-\frac{k-1}{2}-s})^2}
{(-m-\frac{k-1}{2}-s)^3}.
\end{split}
\end{equation*}
Thus we complete the proof of the lemma.
\end{proof}

\begin{lemma}\label{lemma 2}
For $s=\sigma+\i t$, $s'=\sigma'+\i t'$ such that $1/2\leq \sigma, \sigma'\leq 10$, $s\neq \rho$, $s'\neq \rho$,
we have
\bna\label{Im}
\Im\left(\frac{L'}{L}(s,f\otimes\chi)-\frac{L'}{L}(s',f\otimes\chi)\right)
=\Im\sum_\rho\left(\frac{1}{s-\rho}-\frac{1}{s'-\rho}\right)+O(1),
\ena
and
\bea\label{Re}
\Re\frac{L'}{L}(s,f\otimes\chi)
=\sum_\rho\frac{\sigma-\beta}{(\sigma-\beta)^2+(t-\gamma)^2}
+O(\log(q(|t|+3))).
\eea
\end{lemma}
\begin{proof}
By the Hadamard factorization theorem of the entire function $\Lambda(s,f\otimes\chi)$, we have
\bna
\frac{L'}{L}(s,f\otimes\chi)=b(f\otimes\chi)
+\sum_\rho\left(\frac{1}{s-\rho}+\frac{1}{\rho}\right)
-\frac{\Gamma'}{\Gamma}\left(s+\frac{k-1}{2}\right)
-\log\left(\frac{q\sqrt{r}}{2\pi}\right),
\ena
for some $b(f\otimes\chi)\in \mathbb{C}$ with $\Re(b(f\otimes\chi))=-\Re\sum_\rho\frac{1}{\rho}$ (see \cite[Proposition 5.7 (3)]{IK}).
Note that for $z\not\in -\mathbb{N}$,
\bea\label{log derivative gamma}
\frac{\Gamma'}{\Gamma}(z)=-\gamma+\sum_{m\geq 1}\left(\frac{1}{m}-\frac{1}{m-1+z}\right),
\eea
where $\gamma$ is the Euler constant.
Using \eqref{log derivative gamma}, we get that for $\frac{1}{4}\leq u$,
\bna
\Im\frac{\Gamma'}{\Gamma}(u+\i v)=\sum_{m\geq 1}\frac{v}{(m+u-1)^2+v^2}\ll 1.
\ena
Combining with the asymptotic formula of the logarithmic derivative of $\Gamma(z)$,
\bna\label{derivative gamma spectral}
\frac{\Gamma'}{\Gamma}(z)=\log z+O_\varepsilon\left(\frac{1}{|z|}\right),
\quad\text{for}\,\, |\arg z|\leq \pi-\varepsilon,
\ena
we complete the proof of lemma.
\end{proof}

Let $x\geq 4$, we define
\bea\label{sigma_x}
\sigma_x=\sigma_{x,f\otimes\chi,t}
:=\frac{1}{2}+2\max_\rho\left\{\left|\beta-\frac{1}{2}\right|, \frac{5}{\log x}\right\},
\eea
where $\rho=\beta+\i \gamma$ runs through the zeros of $L(s,f\otimes\chi)$ for which
\bea\label{condition 1}
|t-\gamma|\leq\frac{x^{3|\beta-1/2|}}{\log x}.
\eea

We have the following lemma.
\begin{lemma}\label{lemma 3}
Let $x\geq 4$. For $\sigma\geq \sigma_x$, we have
\begin{equation*}
\begin{split}
\frac{L'}{L}(\sigma+\i t,f\otimes\chi)
=&-\sum_{n\leq x^3} \frac{C_f(n)\Lambda_x(n)\chi(n)}{n^{\sigma+\i t}}
+O\left(x^{1/4-\sigma/2}\left|\sum_{n\leq x^3} \frac{C_f(n)\Lambda_x(n)\chi(n)}{n^{\sigma_x+\i t}}\right|\right)\\
&+O\left(x^{1/4-\sigma/2}\log(q(|t|+3))\right)
\end{split}
\end{equation*}
and
\bna
\sum_\rho\frac{\sigma_x-1/2}{(\sigma_x-\beta)^2+(t-\gamma)^2}
=O\left(\left|\sum_{n\leq x^3} \frac{C_f(n)\Lambda_x(n)
\chi(n)}{n^{\sigma_x+\i t}}\right|\right)
+O\big(\log(q(|t|+3))\big).
\ena
\end{lemma}
\begin{proof}
This proof follows \cite[Lemma 4.3]{LS} closely, and we give the details here for completeness.
By \eqref{Re}, we have
\bea\label{Re sigma_x}
\Re\frac{L'}{L}(\sigma_x+\i t,f\otimes\chi)
=\sum_\rho\frac{\sigma_x-\beta}{(\sigma_x-\beta)^2+(t-\gamma)^2}
+O(\log(q(|t|+3))).
\eea
Moreover, if $\rho=\beta+\i \gamma$ is a zero of $L(s,f\otimes\chi)$, then $(1-\beta)+\i \gamma$ is also a zero of $L(s,f\otimes\chi)$.
Thus we have
\begin{equation}
\begin{split}\label{equation 1}
&\sum_\rho\frac{\sigma_x-\beta}{(\sigma_x-\beta)^2+(t-\gamma)^2}\\
&=\frac{1}{2}\left(\sum_\rho\frac{\sigma_x-\beta}
{(\sigma_x-\beta)^2+(t-\gamma)^2}
+\sum_\rho\frac{\sigma_x-(1-\beta)}
{(\sigma_x-(1-\beta))^2+(t-\gamma)^2}\right)\\
&=\left(\sigma_x-\frac{1}{2}\right)
\sum_\rho\frac{(\sigma_x-1/2)^2-(\beta-1/2)^2+(t-\gamma)^2}
{((\sigma_x-\beta)^2+(t-\gamma)^2)((\sigma_x-1+\beta)^2+(t-\gamma)^2)}.
\end{split}
\end{equation}
If $|\beta-1/2|\leq \frac{1}{2}(\sigma_x-1/2)$, then
\begin{equation*}
\begin{split}
(\sigma_x-1/2)^2-(\beta-1/2)^2
\geq&\,\frac{1}{2}\big((\sigma_x-1/2)^2+(\beta-1/2)^2\big)\\
=&\,\frac{1}{4}\big((\sigma_x-\beta)^2+(\sigma_x-1+\beta)^2\big).
\end{split}
\end{equation*}
Thus,
\bea\label{bound 1}
(\sigma_x-1/2)^2-(\beta-1/2)^2+(t-\gamma)^2
\geq\frac{1}{4}\big((\sigma_x-1+\beta)^2+(t-\gamma)^2\big).
\eea
If $|\beta-1/2|>\frac{1}{2}(\sigma_x-1/2)$, then by \eqref{sigma_x} and \eqref{condition 1} we have
\bna
|t-\gamma|>\frac{x^{3|\beta-1/2|}}{\log x}>3|\beta-1/2|.
\ena
Therefore,
\begin{equation}
\begin{split}\label{bound 2}
&(\sigma_x-1/2)^2-(\beta-1/2)^2+(t-\gamma)^2\nonumber\\
&=\big((\sigma_x-1/2)^2+(\beta-1/2)^2\big)
+(t-\gamma)^2-2(\beta-1/2)^2\nonumber\\
&\geq\frac{1}{2}\big((\sigma_x-\beta)^2
+(\sigma_x-1+\beta)^2\big)+\frac{7}{9}(t-\gamma)^2\nonumber\\
&\geq\frac{1}{4}\big((\sigma_x-1+\beta)^2+(t-\gamma)^2\big).
\end{split}
\end{equation}
Combining \eqref{bound 1}, \eqref{bound 2} and \eqref{equation 1}, we have
\bna
\sum_\rho\frac{\sigma_x-\beta}{(\sigma_x-\beta)^2+(t-\gamma)^2}
\geq\frac{1}{4}(\sigma_x-1/2)\sum_\rho\frac{1}
{(\sigma_x-\beta)^2+(t-\gamma)^2}.
\ena
Using this bound and \eqref{Re sigma_x}, we obtain
\bea\label{step 1}
\sum_\rho\frac{\sigma_x-1/2}{(\sigma_x-\beta)^2+(t-\gamma)^2}
\leq 4\left|\frac{L'}{L}(\sigma_x+\i t,f\otimes\chi)\right|
+O\left(\log(q(|t|+3))\right).
\eea
On the other hand, by Lemma \ref{lemma 1}, we have
\begin{equation}
\begin{split}\label{step 2}
\frac{L'}{L}(\sigma+\i t,f\otimes\chi)
=&-\sum_{n\leq x^3} \frac{C_f(n)\Lambda_x(n)\chi(n)}{n^{\sigma+\i t}}
+\frac{w(x,\sigma,t)}{\log^2x}
\sum_\rho\frac{x^{\beta-\sigma}(1+x^{\beta-\sigma})^2}
{\big((\sigma-\beta)^2+(t-\gamma)^2\big)^{3/2}}\\
&+O\left(\frac{x^{-\sigma}}{\log^2x}\right)
\end{split}\end{equation}
with $|w(x,\sigma,t)|\leq1$.
Next we claim that
\bea\label{claim}
\frac{x^{\beta-\sigma}(1+x^{\beta-\sigma})^2}
{\big((\sigma-\beta)^2+(t-\gamma)^2\big)^{3/2}}
\leq 2\log x\frac{x^{1/4-\sigma/2}}{(\sigma_x-\beta)^2+(t-\gamma)^2}.
\eea
In fact, if $\beta\leq \frac{1}{2}(\sigma_x-1/2)$, then we have
\begin{equation*}
\begin{split}
\frac{x^{\beta-\sigma}(1+x^{\beta-\sigma})^2}
{\big((\sigma-\beta)^2+(t-\gamma)^2\big)^{3/2}}
&\leq \frac{4x^{1/4-\sigma/2}}
{(\sigma_x-\beta)\big((\sigma_x-\beta)^2+(t-\gamma)^2\big)}\\
&\leq \frac{8}{\sigma_x-1/2}
\frac{x^{1/4-\sigma/2}}{(\sigma_x-\beta)^2+(t-\gamma)^2}\\
&\leq \frac{4}{5}\log x\frac{x^{1/4-\sigma/2}}
{(\sigma_x-\beta)^2+(t-\gamma)^2}.
\end{split}\end{equation*}
If $\beta>\frac{1}{2}(\sigma_x-1/2)$, then we have
\bna
|t-\gamma|>\frac{x^{3|\beta-1/2|}}{\log x}>3|\beta-1/2|\geq 3|\beta-\sigma_x|.
\ena
Thus, $(t-\gamma)^2>\frac{8}{9}\big((\beta-\sigma_x)^2+(t-\gamma)^2\big)$.
Hence
\begin{equation*}
\begin{split}
\frac{x^{\beta-\sigma}(1+x^{\beta-\sigma})^2}{\big((\sigma-\beta)^2+(t-\gamma)^2\big)^{3/2}}
&\leq\frac{x^{\beta-\sigma}(1+x^{\beta-1/2})^2}{|t-\gamma|(t-\gamma)^2}\\
&\leq\frac{\log x}{x^{3|\beta-1/2|}}\frac{9}{8}\frac{x^{\beta-\sigma}(1+x^{\beta-1/2})^2}{(\beta-\sigma_x)^2+(t-\gamma)^2}\\
&=\frac{9}{8}(\log x)(1+x^{-(\beta-1/2)})^2\frac{x^{1/2-\sigma}}{(\sigma_x-\beta)^2+(t-\gamma)^2}\\
&\leq \frac{9}{8}(\log x)(1+e^{-5})^2\frac{x^{1/2-\sigma}}{(\sigma_x-\beta)^2+(t-\gamma)^2}\\
&<2\log x\frac{x^{{1/4-\sigma/2}}}{(\sigma_x-\beta)^2+(t-\gamma)^2}.
\end{split}\end{equation*}
Using \eqref{claim} and \eqref{step 1}, we have
\begin{equation*}
\begin{split}
&\sum_\rho\frac{x^{\beta-\sigma}(1+x^{\beta-\sigma})^2}{\big((\sigma-\beta)^2+(t-\gamma)^2\big)^{3/2}}\\
&\leq\,8\log x\frac{x^{1/4-\sigma/2}}{\sigma_x-1/2}
\left|\frac{L'}{L}(\sigma_x+\i t,f\otimes\chi)\right|
+O\left(\frac{(\log x)x^{1/4-\sigma/2}\log(q(|t|+3))}{\sigma_x-1/2}\right)\\
&\leq\,\frac{4}{5}(\log x)^2x^{1/4-\sigma/2}
\left|\frac{L'}{L}(\sigma_x+\i t,f\otimes\chi)\right|
+O\left((\log x)^2x^{1/4-\sigma/2}\log(q(|t|+3))\right).
\end{split}
\end{equation*}
Inserting this bound into \eqref{step 2}, we have
\begin{equation}
\begin{split}\label{step 3}
\frac{L'}{L}(\sigma+\i t,f\otimes\chi)
=&-\sum_{n\leq x^3} \frac{C_f(n)\Lambda_x(n)\chi(n)}{n^{\sigma+\i t}}+\frac{4}{5}w'(x,\sigma,t)x^{1/4-\sigma/2}
\left|\frac{L'}{L}(\sigma_x+\i t,f\otimes\chi)\right|\\
&+O\left(x^{1/4-\sigma/2}\log(q(|t|+3))\right),
\end{split}
\end{equation}
with $|w'(x,\sigma,t)|\leq 1$.
By taking $\sigma=\sigma_x$,
\begin{equation*}
\begin{split}
&\left(1-\frac{4}{5}w'(x,\sigma_x,t)x^{1/4-\sigma_x/2}\right)
\frac{L'}{L}(\sigma_x+\i t,f\otimes\chi)\\
&=O\left(\left|\sum_{n\leq x^3}
\frac{C_f(n)\Lambda_x(n)\chi(n)}{n^{\sigma_x+\i t}}\right|\right)
+O\left(x^{1/4-\sigma_x/2}\log(q(|t|+3))\right).
\end{split}\end{equation*}
Since $\left|1-\frac{4}{5}w'(x,\sigma_x,t)x^{1/4-\sigma_x/2}\right|
\geq 1-\frac{4}{5}\geq \frac{1}{5}$, we have
\bea\label{step 4}
\frac{L'}{L}(\sigma_x+\i t,f\otimes\chi)
=O\left(\left|\sum_{n\leq x^3}
\frac{C_f(n)\Lambda_x(n)\chi(n)}{n^{\sigma_x+\i t}}\right|\right)
+O\left(x^{1/4-\sigma_x/2}\log(q(|t|+3))\right).
\eea
Inserting \eqref{step 4} into \eqref{step 3} and \eqref{step 1},
respectively, we complete the proof of the lemma.
\end{proof}

The following proposition provides an approximation of $S(t,f\otimes\chi)$.
\begin{proposition}\label{theorem 1}
For $t\neq 0$, and $x\geq 4$, we have
\begin{equation*}
\begin{split}
S(t,f\otimes\chi)
=&\frac{1}{\pi}\Im\sum_{n\leq x^3}\frac{C_f(n)\Lambda_x(n)\chi(n)}{n^{\sigma_x+\i t}\log n}+O\left((\sigma_x-1/2)\left|\sum_{n\leq x^3} \frac{C_f(n)\Lambda_x(n)\chi(n)}{n^{\sigma_x+\i t}}\right|\right)\\
&+O\big((\sigma_x-1/2)\log(q(|t|+3))\big),
\end{split}\end{equation*}
where $\sigma_x$ is defined as in \eqref{sigma_x}.
\end{proposition}
\begin{proof}
By the definition of $S(t,f\otimes\chi)$, we have
\begin{equation*}
\begin{split}
\pi S(t,f\otimes\chi)
=&-\int_{1/2}^\infty\Im \frac{L'}{L}(\sigma+\i t,f\otimes\chi)\mathrm{d}\sigma\\
=&-\int_{\sigma_x}^\infty\Im \frac{L'}{L}(\sigma+\i t,f\otimes\chi)\mathrm{d}\sigma
-(\sigma_x-1/2)\Im \frac{L'}{L}(\sigma_x+\i t,f\otimes\chi)\\
&+\int_{1/2}^{\sigma_x}\Im\left( \frac{L'}{L}(\sigma_x+\i t,f\otimes\chi)
-\frac{L'}{L}(\sigma+\i t,f\otimes\chi)\right)\mathrm{d}\sigma\\
=:&\, J_1+J_2+J_3.
\end{split}
\end{equation*}
By Lemma \ref{lemma 3}, we have
\begin{equation*}
\begin{split}
J_1
=&-\int_{\sigma_x}^\infty\Im\frac{L'}{L}(\sigma+\i t,f\otimes\chi)
\mathrm{d}\sigma\\
=&\int_{\sigma_x}^\infty\Im\sum_{n\leq x^3} \frac{C_f(n)\Lambda_x(n)\chi(n)}{n^{\sigma+\i t}}\mathrm{d}\sigma
+O\left(\left|\sum_{n\leq x^3} \frac{C_f(n)\Lambda_x(n)\chi(n)}{n^{\sigma_x+\i t}}\right|
\int_{\sigma_x}^\infty x^{1/4-\sigma/2}\mathrm{d}\sigma\right)\\
&+O\left(\log(q(|t|+3))\int_{\sigma_x}^\infty x^{1/4-\sigma/2}\mathrm{d}\sigma\right)\\
=&\Im\sum_{n\leq x^3} \frac{C_f(n)\Lambda_x(n)\chi(n)}{n^{\sigma_x+\i t}\log n}
+O\left(\frac{1}{\log x}\left|\sum_{n\leq x^3} \frac{C_f(n)\Lambda_x(n)\chi(n)}{n^{\sigma_x+\i t}}\right|\right)\\
&+O\left(\frac{\log(q(|t|+3))}{\log x}\right).
\end{split}\end{equation*}
Taking $\sigma=\sigma_x$ in Lemma \ref{lemma 3}, we have
\begin{equation*}
\begin{split}
J_2&\leq(\sigma_x-1/2)
\left|\frac{L'}{L}(\sigma_x+\i t,f\otimes\chi)\right|\\
&\ll(\sigma_x-1/2)\left|\sum_{n\leq x^3} \frac{C_f(n)\Lambda_x(n)\chi(n)}{n^{\sigma_x+\i t}}\right|
+(\sigma_x-1/2)\log(q(|t|+3)).
\end{split}\end{equation*}
We apply Lemma \ref{lemma 2} to bound $J_3$ and get
\begin{equation*}
\begin{split}
&\Im\left( \frac{L'}{L}(\sigma_x+\i t,f\otimes\chi)
-\frac{L'}{L}(\sigma+\i t,f\otimes\chi)\right)\\
&=\sum_{\rho}\Im\left(\frac{1}{(\sigma_x-\beta)+\i (t-\gamma)}
-\frac{1}{(\sigma-\beta)+\i (t-\gamma)}\right)+O(1)\\
&=\sum_\rho \frac{(\sigma_x-\sigma)(\sigma+\sigma_x-2\beta)(t-\gamma)}
{\big((\sigma_x-\beta)^2+(t-\gamma)^2\big)
\big((\sigma-\beta)^2+(t-\gamma)^2\big)}+O(1).
\end{split}\end{equation*}
Hence
\begin{equation*}
\begin{split}
|J_3|
&\leq\sum_\rho\int_{1/2}^{\sigma_x}
\frac{(\sigma_x-\sigma)|\sigma+\sigma_x-2\beta||t-\gamma|}
{\big((\sigma_x-\beta)^2+(t-\gamma)^2\big)
\big((\sigma-\beta)^2+(t-\gamma)^2\big)}\mathrm{d}\sigma
+O(\sigma_x-1/2)\\
&\leq\sum_\rho\frac{\sigma_x-1/2}
{(\sigma_x-\beta)^2+(t-\gamma)^2}\int_{1/2}^{\sigma_x}
\frac{|\sigma+\sigma_x-2\beta||t-\gamma|}
{(\sigma-\beta)^2+(t-\gamma)^2}\mathrm{d}\sigma
+O(\sigma_x-1/2).
\end{split}\end{equation*}
If $|\beta-1/2|\leq\frac{1}{2}(\sigma_x-1/2)$, then for $1/2\leq\sigma\leq\sigma_x$,
\begin{equation*}
\begin{split}
|\sigma+\sigma_x-2\beta|
&=|(\sigma-1/2)+(\sigma_x-1/2)-2(\beta-1/2)|\\
&\leq |\sigma-1/2|+|\sigma_x-1/2|+2|\beta-1/2|\leq 3|\sigma_x-1/2|.
\end{split}\end{equation*}
Thus,
\begin{equation*}
\begin{split}
\int_{1/2}^{\sigma_x}\frac{|\sigma+\sigma_x-2\beta||t-\gamma|}
{(\sigma-\beta)^2+(t-\gamma)^2}\mathrm{d}\sigma
&\leq 3(\sigma_x-1/2)\int_{-\infty}^\infty\frac{|t-\gamma|}
{(\sigma-\beta)^2+(t-\gamma)^2}\mathrm{d}\sigma\\
&\leq 10(\sigma_x-1/2).
\end{split}\end{equation*}
If $|\beta-1/2|>\frac{1}{2}(\sigma_x-1/2)$, then
\bna
|t-\gamma|>\frac{x^{3|\beta-1/2|}}{\log x}>3|\beta-1/2|
\ena
and for $1/2\leq\sigma\leq\sigma_x$,
\bna
|\sigma+\sigma_x-2\beta|\leq |\sigma-1/2|+|\sigma_x-1/2|+2|\beta-1/2|\leq 6|\beta-1/2|.
\ena
Thus,
\begin{equation*}
\begin{split}
\int_{1/2}^{\sigma_x}\frac{|\sigma+\sigma_x-2\beta||t-\gamma|}
{(\sigma-\beta)^2+(t-\gamma)^2}\mathrm{d}\sigma
&<\int_{1/2}^{\sigma_x}\frac{|\sigma+\sigma_x-2\beta|}
{|t-\gamma|}\mathrm{d}\sigma\\
&\leq \int_{1/2}^{\sigma_x}\frac{6|\beta-1/2|}
{3|\beta-1/2|}\mathrm{d}\sigma=2(\sigma_x-1/2).
\end{split}\end{equation*}
By Lemma \ref{lemma 3}, we obtain
\begin{equation*}
\begin{split}
|J_3|&\leq10(\sigma_x-1/2)
\sum_\rho\frac{\sigma_x-1/2}{(\sigma_x-\beta)^2+(t-\gamma)^2}
+O(\sigma_x-1/2)\\
&=O\left((\sigma_x-1/2)\left|\sum_{n\leq x^3} \frac{C_f(n)\Lambda_x(n)\chi(n)}{n^{\sigma_x+\i t}}\right|\right)
+O\big((\sigma_x-1/2)\log(q(|t|+3))\big).
\end{split}\end{equation*}
The proposition follows from these bounds for $J_1$, $J_2$ and $J_3$.
\end{proof}

\begin{remark}
As a consequence of this proposition, under GRH, we can obtain an
analogue result of Selberg \cite{Selberg2}.
In fact, we have $\sigma_x=\frac{1}{2}+\frac{10}{\log x}$ under GRH.
Thus by \eqref{eq:KS} we have
\bna
|C_f(p^m)\Lambda_x(p^m)\chi(p^m)|\ll \log p.
\ena
Thus
\bna
\left|\Im\sum_{n\leq x^3}\frac{C_f(n)\Lambda_x(n)\chi(n)}{n^{\sigma_x+\i t}\log n}\right|
\ll \sum_{p\leq x^3}p^{-1/2}
\ll \frac{x^{3/2}}{\log x},
\ena
and
\bna
(\sigma_x-1/2)\left|\sum_{n\leq x^3}
\frac{C_f(n)\Lambda_x(n)\chi(n)}{n^{\sigma_x+\i t}}\right|
\ll \frac{1}{\log x}\sum_{p\leq x^3}p^{-1/2}\log p\ll \frac{x^{3/2}}{\log x}.
\ena
By taking $x=\big(\log(q(|t|+3))\big)^{2/3}$ in Proposition
\ref{theorem 1}, we conclude that
\bna
S(t,f\otimes\chi)
\ll\frac{\log(q(|t|+3))}
{\log\log(q(|t|+3))}.
\ena
\end{remark}

\section{Proof of proposition \ref{main propsition}}\label{Proof of the main theorem}
In this section, we give the proof of Proposition \ref{main propsition}.
\subsection{Proof of the asymptotic formula \eqref{M(t,f) moment}}
Firstly, we write
\begin{equation}\label{eq:M}
\begin{split}
\sideset{}{^*}\sum_{\chi \bmod q}M(t,f\otimes\chi)^{2n}
=&\;\;\sum_{\chi \bmod q}M(t,f\otimes\chi)^{2n}
-\left(\frac{1}{\pi}\Im\sum_{p\leq x^3}\frac{C_f(p)\chi_0(p)}{p^{1/2+\i t}}
\right)^{2n},
\end{split}\end{equation}
where the summation on the right-hand side runs over all characters $\chi \bmod q$ and $\chi_0$ denotes the principal character modulus $q$.
Recall that $C_f(p)=\lambda_f(p)$ and we further have
\bna
M(t,f\otimes\chi)=\frac{1}{\pi}\Im\sum_{p\leq x^3}\frac{C_f(p)\chi(p)}{p^{1/2+\i t}}
=\frac{-\i}{2\pi}\left(\mathcal{S}-\overline{\mathcal{S}}\right),
\ena
where
\bna
\mathcal{S}=\sum_{p\leq x^3}\frac{\lambda_f(p)\chi(p)}{p^{1/2+\i t}}.
\ena
Set $x=q^{\delta/3}$ for a suitably small $\delta>0$ to be chosen later. Thus we have
\bea\label{eq:M2}
\sum_{\chi \bmod q}M(t,f\otimes\chi)^{2n}
=\frac{(-\i)^{2n}}{(2\pi)^{2n}}
\sum_{m=0}^{2n}(-1)^{2n-m}\binom{2n}{m}
\sum_{\chi \bmod q}\mathcal{S}^{m}\overline{\mathcal{S}}^{2n-m}.
\eea
Note that
\begin{equation}\label{general term}
\begin{split}
\sum_{\chi \bmod q}\mathcal{S}^{m}\overline{\mathcal{S}}^{2n-m}
=&\sum_{p_1,\cdots,p_{2n}\leq x^3}\frac{\lambda_f(p_1)
\cdots\lambda_f(p_{2n})}{\sqrt{p_{1}p_{2}\cdots p_{2n}}}
\left(\frac{p_{m+1}\cdots p_{2n}}{p_{1}p_{2}\cdots p_{m}}\right)^{\i t}\\
&\qquad\qquad\qquad\times\sum_{\chi \bmod q}\chi\left(p_{1}p_{2}\cdots p_{m}\right)
\overline{\chi}\left(p_{m+1}\cdots p_{2n}\right).
\end{split}\end{equation}
By the orthogonality of Dirichlet characters, we have
\begin{align*}
\sum_{\chi \bmod q} \chi(m)\overline{\chi}(n) = \begin{cases}
\varphi(q), & \text{if } m \equiv n \bmod q, \\
0, & \text{otherwise}.
\end{cases}
\end{align*}
Hence, the sum in the right-side hand of \eqref{general term} vanishes unless
$p_1 \cdots p_m \equiv p_{m+1} \cdots p_{2n} \pmod{q}$.
Assume $0<\delta<1/n$.
Note that $p_i\leq x^3=q^{\delta}$ for $i=1,2,\cdots,2n$.
It follows that only the diagonal term $m=n$ survives.
Thus we have
\begin{equation*}
\begin{split}
\sum_{\chi \bmod q}|\mathcal{S}|^{2n}
=&\,\varphi(q)\sum_{p_1,\cdots,p_{2n}\leq x^3\atop p_{1}p_{2}\cdots p_{n}=p_{n+1}\cdots p_{2n}}
\frac{\lambda_f(p_1)\cdots\lambda_f(p_n)\lambda_f(p_{n+1})\cdots\lambda_f(p_{2n})}{p_{1}p_{2}\cdots p_{n}}\\
=&\,n!\varphi(q)\sum_{p_1,\cdots,p_{n}\leq x^3 \atop p_i \,\text{distinct}}
\prod_{i=1}^{n}\frac{|\lambda_f(p_i)|^2}{p_i}
+O\left(\varphi(q)\sum_{p_1,\cdots,p_{n-1}\leq x^3}\frac{\lambda_f^2(p_1)
\cdots\lambda_f^2(p_{n-2})\lambda_f^4(p_{n-1})}{p_{1}\cdots p_{n-2}p_{n-1}^2}\right)\\
=&\,n!\varphi(q)\left(\sum_{p\leq x^3}\frac{|\lambda_f(p)|^2}{p}\right)^{n}
+O\left(\varphi(q)\left(\sum_{p\leq x^3}\frac{|\lambda_f(p)|^2}{p}\right)^{n-2}\right).
\end{split}\end{equation*}
By Lemma \ref{bound} (3), we have
\bna
\sum_{\chi \bmod q}|\mathcal{S}|^{2n}
=n!\varphi(q)\left(\log\log q \right)^{n}
+O\left(\varphi(q)\left(\log\log q\right)^{n-1}\right).
\ena
Inserting this into \eqref{eq:M2}, we get
\bea\label{eq:M3}
\sum_{\chi \bmod q}|M(t,f\otimes\chi)|^{2n}
=\frac{(2n)!}{n!(2\pi)^{2n}}\varphi(q)\left(\log\log q\right)^{n}
+O\left(\varphi(q)\left(\log\log q\right)^{n-1}\right).
\eea
It remains to bound the contribution from the principal character in \eqref{eq:M}.
Notice that
\bna
\Bigg(\frac{1}{\pi}\Im\sum_{p\leq x^3}\frac{C_f(p)\chi_0(p)}{p^{1/2+\i t}}
\Bigg)^{2n}
\leq\left(\frac{1}{\pi}\sum_{m\leq x^3}\frac{1}{\sqrt{m}}\right)^{2n}
=O(x^{3n})
=O(q),
\ena
since $0<\delta<1/n$.
Combining this with \eqref{eq:M} and \eqref{eq:M3},
we complete the proof of the asymptotic formula in \eqref{M(t,f) moment}.
\subsection{Proof of \eqref{R(t,f) moment}}
To prove \eqref{R(t,f) moment},
we first present two crucial preliminary lemmas.
\begin{lemma}\label{lemma 6}
Let $t>0$ be given. For prime number $q$ sufficiently large and $x=q^{\delta/3}$ with
$0<\delta<\frac{3\eta}{8n+3A+3}$, we have
\bna
\frac{1}{\varphi^*(q)}\;\;\sideset{}{^*}\sum_{\chi \bmod q}
(\sigma_{x}-1/2)^{4n}x^{4n(\sigma_{x}-1/2)}
\ll_{t,n,\delta,A}\frac{1}{(\log q)^{4n}},
\ena
where $\sigma_{x}$ is defined by \eqref{sigma_x}, $\eta$ and $A$ are as in Theorem \ref{zero-density}.
\end{lemma}
\begin{proof}
From the definition of $\sigma_{x}$, we have
\begin{equation*}
\begin{split}
&(\sigma_{x}-1/2)^{4n}x^{4n(\sigma_{x}-1/2)}\\
\leq&\left(\frac{10}{\log x}\right)^{4n}x^{40n/\log x}
+2^{4n+1}\sum_{\beta>\frac{1}{2}+\frac{5}{\log x}\atop |t-\gamma|\leq\frac{x^{3|\beta-1/2|}}{\log x}}(\beta-1/2)^{4n}x^{8n(\beta-1/2)}.
\end{split}\end{equation*}
On the other hand,
\begin{equation*}
\begin{split}
&\sum_{\beta>\frac{1}{2}+\frac{5}{\log x}\atop |t-\gamma|\leq\frac{x^{3|\beta-1/2|}}{\log x}}(\beta-1/2)^{4n}x^{8n(\beta-1/2)}\\
&\leq\sum_{j=5}^{\frac{1}{2}\lfloor\log x\rfloor}\left(\frac{j+1}{\log x}\right)^{4n}
x^{8n\frac{j+1}{\log x}}\sum_{\frac{1}{2}+\frac{j}{\log x}<\beta\leq\frac{1}{2}+\frac{j+1}{\log x}\atop |t-\gamma|\leq\frac{x^{3|\beta-1/2|}}{\log x}}1\\
&\leq\frac{1}{(\log x)^{4n}}\sum_{j=5}^{\frac{1}{2}\lfloor\log x\rfloor}\left(j+1\right)^{4n}
e^{8n(j+1)}N\left(f\otimes\chi;\frac{1}{2}+\frac{j}{\log x}, t-\frac{e^{3(j+1)}}{\log x},t+\frac{e^{3(j+1)}}{\log x}\right).
\end{split}\end{equation*}

By Theorem \ref{zero-density}, there exists an absolute constant $A>0$ such that
\begin{equation*}
\begin{split}
&\frac{1}{\varphi^*(q)}\;\;\sideset{}{^*}\sum_{\chi \bmod q}
\sum_{\beta>\frac{1}{2}+\frac{5}{\log x}\atop |t-\gamma|\leq\frac{x^{3|\beta-1/2|}}{\log x}}(\beta-1/2)^{4n}x^{8n(\beta-1/2)}\\
&\leq\frac{1}{(\log x)^{4n}}\sum_{j=5}^{\frac{1}{2}\lfloor\log x\rfloor}\left(j+1\right)^{4n}
e^{8n(j+1)}\frac{1}{\varphi^*(q)}\;\;\sideset{}{^*}\sum_{\chi \bmod q}
N\left(f\otimes\chi;\frac{1}{2}+\frac{j}{\log x}, t-\frac{e^{3(j+1)}}{\log x},t+\frac{e^{3(j+1)}}{\log x}\right)\\
&\ll\frac{1}{(\log x)^{4n}}\sum_{j=5}^{\infty}\left(j+1\right)^{4n}
e^{8n(j+1)}\left(1+|t|+\frac{e^{3(j+1)}}{\log x}\right)^A q^{-\eta\frac{j}{\log x}}\frac{e^{3(j+1)}}{\log x}\log q\\
&\ll_{t,n,A,\delta}\,\frac{\log q}{(\log x)^{4n+A+1}}\sum_{j=5}^{\infty}\left(j+1\right)^{4n}
e^{\big(8n+3A+3-\frac{3\eta}{\delta}\big)j}\\
&\ll_{t,n,A,\delta}\,\frac{1}{(\log q)^{4n}}
\end{split}\end{equation*}
provided that $0<\delta<\frac{3\eta}{8n+3A+3}$.
In addition, we have
\bna
\frac{1}{\varphi^*(q)}\;\;\sideset{}{^*}\sum_{\chi \bmod q}\left(\frac{10}{\log x}\right)^{4n}x^{40n/\log x}
\ll\frac{1}{(\log x)^{4n}}\ll\frac{1}{(\log q)^{4n}}.
\ena
Thus we complete the proof of the lemma.
\end{proof}

\begin{lemma}\label{lemma 5}
We have
\begin{equation*}
\begin{split}
R(t,f\otimes\chi)
=&\,O\left(\left|\Im\sum_{p\leq x^3}
\frac{C_f(p)(\Lambda_x(p)-\Lambda(p))\chi(p)}{p^{1/2+\i t}\log p}\right|\right)
+O\left(\left|\Im\sum_{p\leq x^{3/2}}\frac{C_f(p^2)\Lambda_x(p^2)\chi(p^2)}{p^{1+2\i t}\log p}\right|\right)\\
&+O\left((\sigma_x-1/2)x^{\sigma_x-1/2}\int_{1/2}^\infty
x^{1/2-\sigma}\left|\sum_{p\leq x^3}\frac{C_f(p)\Lambda_x(p)\chi(p)\log(xp)}{p^{\sigma+\i t}}\right|\mathrm{d}\sigma\right)\\
&+O\big((\sigma_x-1/2)\log(q(|t|+3))\big)+O(1).
\end{split}\end{equation*}
\end{lemma}
\begin{proof}
By Proposition \ref{theorem 1}, we have
\begin{equation*}
\begin{split}
&R(t,f\otimes\chi)\\
=&\frac{1}{\pi}\Im\sum_{p\leq x^3}\frac{C_f(p)(\Lambda_x(p)p^{1/2-\sigma_x}-\Lambda(p))\chi(p)}{p^{1/2+\i t}\log p}+\frac{1}{\pi}\Im\sum_{m=2}^\infty\sum_{p^m\leq x^3}\frac{C_f(p^m)\Lambda_x(p^m)\chi(p^m)}{p^{m(\sigma_x+\i t)}\log p^m}\\
&+O\left((\sigma_x-1/2)\left|\sum_{m=1}^\infty\sum_{p^m\leq x^3} \frac{C_f(p^m)\Lambda_x(p^m)\chi(p^m)}{p^{m(\sigma_x+\i t)}}\right|\right)
+O\big((\sigma_x-1/2)\log(q(|t|+3))\big).
\end{split}\end{equation*}
Using the estimates in \eqref{eq:KS} and \eqref{Deligne bound}, we have
\bna
\sum_{m=3}^\infty\sum_{p^m\leq x^3}\frac{C_f(p^m)\Lambda_x(p^m)
\chi(p^m)}{p^{m(\sigma_x+\i t)}}=O(1)
\quad\text{and}\quad
\sum_{m=3}^\infty\sum_{p^m\leq x^3}\frac{C_f(p^m)\Lambda_x(p^m)\chi(p^m)}{p^{m(\sigma_x+\i t)}\log p}=O(1).
\ena
Note that
\begin{equation*}
\begin{split}
(\sigma_x-1/2)\left|\sum_{p\leq x^{3/2}} \frac{C_f(p^2)\Lambda_x(p^2)
\chi(p^2)}{p^{2(\sigma_x+\i t)}}\right| &\ll(\sigma_x-1/2)\sum_{p\leq x^{3/2}}\frac{\log p}{p}\\
&\ll(\sigma_x-1/2)\log x\ll (\sigma_x-1/2)\log q.
\end{split}\end{equation*}
Thus,
\begin{equation*}
\begin{split}
&R(t,f\otimes\chi)\\
=&\,\frac{1}{\pi}\Im\sum_{p\leq x^3}\frac{C_f(p)(\Lambda_x(p)-\Lambda(p))\chi(p)}{p^{1/2+\i t}\log p}
-\frac{1}{\pi}\Im\sum_{p\leq x^3}
\frac{C_f(p)\Lambda_x(p)\chi(p)(1-p^{1/2-\sigma_x})}{p^{1/2+\i t}\log p}\\
&+\frac{1}{\pi}\Im\sum_{p\leq x^{3/2}}\frac{C_f(p^2)\Lambda_x(p^2)\chi(p^2)}{p^{1+2\i t}\log p^2}
-\frac{1}{\pi}\Im\sum_{p\leq x^{3/2}}\frac{C_f(p^2)\Lambda_x(p^2)\chi(p^2)(1-p^{1-2\sigma_x})}{p^{1+2\i t}\log p^2}\\
&+O\left((\sigma_x-1/2)\left|\sum_{p\leq x^3} \frac{C_f(p)\Lambda_x(p)\chi(p)}{p^{\sigma_x+\i t}}\right|\right)
+O\big((\sigma_x-1/2)\log(q(|t|+3))\big)+O(1)\\
:=&\sum_{i=1}^{7}G_i.
\end{split}\end{equation*}
Note that for $1/2\leq a\leq \sigma_x$,
\begin{equation*}
\begin{split}
\left|\sum_{p\leq x^3} \frac{C_f(p)\Lambda_x(p)\chi(p)}{p^{a+\i t}}\right|
&=x^{a-1/2}\left|\int_{a}^\infty x^{1/2-\sigma}\sum_{p\leq x^3}
\frac{C_f(p)\Lambda_x(p)\chi(p)\log(xp)}{p^{\sigma+\i t}}\mathrm{d}\sigma\right|\\
&\leq x^{\sigma_x-1/2}\int_{1/2}^\infty x^{1/2-\sigma}\left|\sum_{p\leq x^3}\frac{C_f(p)\Lambda_x(p)\chi(p)\log(xp)}{p^{\sigma+\i t}}\right|\mathrm{d}\sigma.
\end{split}\end{equation*}
Thus,
\begin{equation*}
\begin{split}
G_2\ll&\,\left|\sum_{p\leq x^3}
\frac{C_f(p)\Lambda_x(p)\chi(p)(1-p^{1/2-\sigma_x})}{p^{1/2+\i t}\log p}\right|
=\left|\int_{1/2}^{\sigma_x} \sum_{p\leq x^3} \frac{C_f(p)\Lambda_x(p)\chi(p)}{p^{a+\i t}}\mathrm{d}a\right|\\
\leq&\,(\sigma_x-1/2)x^{\sigma_x-1/2}\int_{1/2}^\infty x^{1/2-\sigma}
\left|\sum_{p\leq x^3}\frac{C_f(p)\Lambda_x(p)\chi(p)\log(xp)}{p^{\sigma+\i t}}\right|\mathrm{d}\sigma,
\end{split}\end{equation*}
and similarly,
\begin{equation*}
\begin{split}
G_5&\ll(\sigma_x-1/2)\left|\sum_{p\leq x^3} \frac{C_f(p)\Lambda_x(p)\chi(p)}{p^{\sigma_x+\i t}}\right|\\
&\leq (\sigma_x-1/2)x^{\sigma_x-1/2}\int_{1/2}^\infty
x^{1/2-\sigma}\left|\sum_{p\leq x^3}\frac{C_f(p)\Lambda_x(p)\chi(p)\log(xp)}{p^{\sigma+\i t}}\right|\mathrm{d}\sigma.
\end{split}\end{equation*}
Moreover,
\begin{equation*}
\begin{split}
G_4&\ll\left|\sum_{p\leq x^{3/2}}\frac{C_f(p^2)\Lambda_x(p^2)\chi(p^2)(1-p^{1-2\sigma_x})}{p^{1+2\i t}\log p}\right|\\
&\ll \sum_{p\leq x^{3/2}}\frac{1}{p}(1-p^{1-2\sigma_x})\ll \sum_{p\leq x^{3/2}}(\sigma_x-1/2)\frac{\log p}{p}
\ll(\sigma_x-1/2)\log q.
\end{split}\end{equation*}
This completes the proof of the lemma.
\end{proof}
Now we are ready to prove \eqref{R(t,f) moment}.
Recall that $\sigma_x=\sigma_{x,f\otimes\chi,t}$ depending on $f,\chi$.
By Lemma \ref{lemma 5}, we have
\bea\label{R}
&&\frac{1}{\varphi^*(q)}\;\;\sideset{}{^*}\sum_{\chi \bmod q}|R(t,f\otimes\chi)|^{2n}\nonumber\\
&&\ll\frac{1}{\varphi^*(q)}\;\;\sideset{}{^*}\sum_{\chi \bmod q}\left|\Im\sum_{p\leq x^3}
\frac{C_f(p)(\Lambda_x(p)-\Lambda(p))\chi(p)}{p^{1/2+\i t}\log p}\right|^{2n}\nonumber\\
&&+\frac{1}{\varphi^*(q)}\;\;\sideset{}{^*}\sum_{\chi \bmod q}
\left|\Im\sum_{p\leq x^{3/2}}\frac{C_f(p^2)\Lambda_x(p^2)\chi(p^2)}{p^{1+2\i t}\log p}\right|^{2n}\nonumber\\
&&+\frac{1}{\varphi^*(q)}\;\;\sideset{}{^*}\sum_{\chi \bmod q}(\sigma_x-1/2)^{2n}x^{2n(\sigma_x-1/2)}\left(\int_{1/2}^\infty
x^{1/2-\sigma}\left|\sum_{p\leq x^3}\frac{C_f(p)\Lambda_x(p)\chi(p)\log(xp)}{p^{\sigma+\i t}}\right|
\mathrm{d}\sigma\right)^{2n}\nonumber\\
&&+\frac{1}{\varphi^*(q)}\;\;\sideset{}{^*}\sum_{\chi \bmod q}(\sigma_x-1/2)^{2n}
\big(\log(q(|t|+3))\big)^{2n}+1.
\eea
Since
\bna
|\Lambda_x(p)-\Lambda(p)|=O\left(\frac{\log^3p}{\log^2x}\right)\quad\text{and}\quad
C_f(p^2)=\lambda_f(p^2)-\chi_r(p),
\ena
the first two terms are $O(1)$ by Lemma \ref{lem18}.
The second-to-last term is of O(1) by Cauchy--Schwarz inequality and Lemma \ref{lemma 6}.
For the third term,
it follows from Cauchy--Schwarz inequality that
\begin{equation*}
\begin{split}
&\frac{1}{\varphi^*(q)}\;\;\sideset{}{^*}\sum_{\chi \bmod q}(\sigma_x-1/2)^{2n}x^{2n(\sigma_x-1/2)}
\left(\int_{1/2}^\infty x^{1/2-\sigma}\left|\sum_{p\leq x^3}\frac{C_f(p)\Lambda_x(p)\chi(p)\log(xp)}{p^{\sigma+\i t}}\right|\mathrm{d}\sigma\right)^{2n}\\
&\leq \; \left(\frac{1}{\varphi^*(q)}\;\;\sideset{}{^*}\sum_{\chi \bmod q}(\sigma_x-1/2)^{4n}x^{4n(\sigma_x-1/2)}\right)^{1/2}\\
&\times\left(\frac{1}{\varphi^*(q)}\;\sideset{}{^*}\sum_{\chi \bmod q}\left(\int_{1/2}^\infty x^{1/2-\sigma}\left|\sum_{p\leq x^3}\frac{C_f(p)\Lambda_x(p)\chi(p)\log(xp)}{p^{\sigma+\i t}}\right|\mathrm{d}\sigma\right)^{4n}\right)^{1/2}.
\end{split}\end{equation*}
By H\"{o}lder's inequality with the exponents $4n/(4n-1)$ and $4n$,  we get
\begin{equation*}
\begin{split}
&\left(\int_{1/2}^\infty x^{1/2-\sigma}\left|\sum_{p\leq x^3}
\frac{C_f(p)\Lambda_x(p)\chi(p)\log(xp)}{p^{\sigma+\i t}}\right|
\mathrm{d}\sigma\right)^{4n}\\
&\leq\left(\int_{1/2}^\infty
x^{1/2-\sigma}\mathrm{d}\sigma\right)^{4n-1}\left(\int_{1/2}^\infty
x^{1/2-\sigma}\left|\sum_{p\leq x^3}\frac{C_f(p)\Lambda_x(p)\chi(p)\log(xp)}{p^{\sigma+\i t}}\right|^{4n}\mathrm{d}\sigma\right)\\
&=\frac{1}{(\log x)^{4n-1}}\int_{1/2}^\infty
x^{1/2-\sigma}\left|\sum_{p\leq x^3}\frac{\Lambda_x(p)C_f(p)\chi(p)\log(xp)}{p^{\sigma+\i t}}\right|^{4n}\mathrm{d}\sigma.
\end{split}
\end{equation*}
Note that
\begin{equation*}
\begin{split}
&\left|\sum_{p\leq x^3}\frac{C_f(p)\Lambda_x(p)\chi(p)\log(xp)}{p^{\sigma+\i t}}\right|^{4n}\\
&\ll_n \left|\Re\sum_{p\leq x^3}
\frac{\lambda_f(p)\Lambda_x(p)\chi(p)\log(xp)}{p^{\sigma+\i t}}\right|^{4n}
+\left|\Im\sum_{p\leq x^3}
\frac{\lambda_f(p)\Lambda_x(p)\chi(p)\log(xp)}{p^{\sigma+\i t}}\right|^{4n}\\
&\ll_n \left(\sum_{p\leq x^3}
\frac{\lambda_f(p)\Lambda_x(p)\chi(p)\log(xp)(p^{1/2-\sigma})}{p^{1/2+\i t}}+\sum_{p\leq x^3}
\frac{\lambda_f(p)\Lambda_x(p)\overline{\chi}(p)\log(xp)(p^{1/2-\sigma})}{p^{1/2-\i t}}\right)^{4n}\\
&\quad+\left(\sum_{p\leq x^3}
\frac{\lambda_f(p)\Lambda_x(p)\chi(p)\log(xp)(p^{1/2-\sigma})}{p^{1/2+\i t}}-\sum_{p\leq x^3}
\frac{\lambda_f(p)\Lambda_x(p)\overline{\chi}(p)\log(xp)(p^{1/2-\sigma})}{p^{1/2-\i t}}\right)^{4n}.
\end{split}
\end{equation*}
By using the same argument as in the proof of \eqref{M(t,f) moment} and Lemma \ref{bound} (2), we get
\begin{equation*}
\begin{split}
&\frac{1}{\varphi^*(q)}\;\sideset{}{^*}\sum_{\chi \bmod q}\left(\int_{1/2}^\infty x^{1/2-\sigma}
\left|\sum_{p\leq x^3}\frac{C_f(p)\Lambda_x(p)\chi(p)\log(xp)}{p^{\sigma+\i t}}
\right|\mathrm{d}\sigma\right)^{4n}\\
&\ll\,\frac{1}{(\log x)^{4n-1}}\int_{1/2}^\infty
x^{1/2-\sigma}\frac{1}{\varphi^*(q)}\;\sideset{}{^*}\sum_{\chi \bmod q}
\left|\sum_{p\leq x^3}\frac{C_f(p)\Lambda_x(p)\chi(p)\log(xp)}{p^{\sigma+\i t}}\right|^{4n}\mathrm{d}\sigma\\
&\ll_{t,n} \frac{1}{(\log x)^{4n-1}}\int_{1/2}^\infty
x^{1/2-\sigma}\left(\sum_{p\leq x^3}\frac{|\lambda_f(p)\Lambda_x(p)\log(xp)p^{1/2-\sigma}|^2}{p}\right)^{2n}\mathrm{d}\sigma\\
&\ll_{t,n} \frac{1}{(\log x)^{4n-1}}\int_{1/2}^\infty
x^{1/2-\sigma}\left((\log x)^2\sum_{p\leq x^3}\frac{(\log p)^2}{p}\right)^{2n}\mathrm{d}\sigma\\
&\ll_{t,n} (\log x)^{4n}.
\end{split}\end{equation*}
Combining this with Lemma \ref{lemma 6}, we conclude that the third term in \eqref{R}
is of $O(1)$.
This completes the proof of \eqref{R(t,f) moment}.

\bibliographystyle{amsplain}

\begin{thebibliography}{10}

\bibitem{Backlund}
R. J. Backlund,
{\it \"{U}ber die Nullstellen der Riemannschen Zetafunktion}. (German)
Acta Math. \textbf{41} (1916), no. 1, 345--375.

\bibitem{Billingsley}
P. Billingsley,
{\it Probability and measure}. (English summary)
Third edition. Wiley Series in Probability and Mathematical Statistics. A Wiley-Interscience Publication.
John Wiley \& Sons, Inc., New York, 1995. xiv+593 pp.

\bibitem{BFKMMS}
V. Blomer, E. Fouvry, E. Kowalski, P. Michel, D. Milicevic and W. Sawin,
{\it The second moment theory of families of $L$-functions---the case of twisted Hecke $L$-functions},
Mem. Amer. Math. Soc. {\bf 282} (2023), no.~1394, v+148 pp.


\bibitem{Cramer}
H. Cram\'{e}r,
{\it \"{U}ber die Nulistellen der Zetafunktion}. (German)
Math. Z. \textbf{2} (1918), no. 3-4, 237--241.

\bibitem{Dav}
H. Davenport, {\it Multiplicative number theory}, third edition,
Graduate Texts in Mathematics, 74, Springer, New York, 2000.

\bibitem{De}
P. Deligne, {\it La conjecture de Weil. I.}  Inst. Hautes \'{E}tudes Sci. Publ. Math. No. \textbf{43} (1974), 273--307.

\bibitem{GM}
L. Guth and J. Maynard, {\it New large value estimates for Dirichlet polynomials}.
arXiv: 2405.20552




\bibitem{HB}
D. R. Heath-Brown, {\it Zero density estimates for the Riemann zeta-function and Dirichlet L-functions}.
J. London Math. Soc. (2) \textbf{19} (1979), no. 2, 221--232.

\bibitem{HL}
D. A. Hejhal and W. Luo, {\it On a spectral analog of Selberg's result on $S(T)$}.
Internat. Math. Res. Notices 1997, no. 3, 135--151.


\bibitem{Hough}
B. Hough, {\it Zero-density estimate for modular form $L$-functions in weight aspect}.
Acta Arith. \textbf{154} (2012), no. 2, 187--216.

\bibitem{Huxley}
M. N. Huxley, {\it On the difference between consecutive primes}. Invent. Math. \textbf{15} (1972), 164--170.

\bibitem{Ingham}
A. E. Ingham, {\it On the estimation of $N(\sigma,T)$}. Quart. J. Math. Oxford Ser. \textbf{11} (1940), 291--292.

\bibitem{IK}
H. Iwaniec and E. Kowalski, {\it Analytic number theory.} American Mathematical Society Colloquium Publications, 53. American Mathematical Society, Providence, RI, 2004. xii+615 pp.

\bibitem{Ivic}
A. Ivi\'{c}, {\it The Riemann zeta-function. Theory and applications}.
Reprint of the 1985 original [Wiley, New York; MR0792089]. Dover Publications, Inc., Mineola, NY, 2003. xxii+517 pp.


\bibitem{Jutila}
M. Jutila, {\it Zero-density estimates for L-functions}. Acta Arith. \textbf{32} (1977), no. 1, 55--62.

\bibitem{K}
E. Kowalski, {\it The rank of the Jacobian of modular curves: Analytic methods}, ProQuest LLC, Ann Arbor, MI, 1998.

\bibitem{KM}
E. Kowalski and P. Michel,
{\it The analytic rank of $J_0(q)$ and zeros of automorphic $L$-functions}.
Duke Math. J. \textbf{100} (1999), no. 3, 503--542.




\bibitem{Littlewood}
J. E. Littlewood,
{\it On the zeros of the Riemann zeta-function}.
Proc. London Math. Soc. \textbf{22}(3), 295--318.

\bibitem{LL}
S.-C. Liu and S. Liu, {\it A $GL_3$ analog of Selberg's result on $S(t)$}. Ramanujan J. \textbf{56} (2021), no. 1, 163--181.

\bibitem{LS}
S.-C. Liu and J. Shim, {\it Moments of $S(t,f)$ associated with holomorphic Hecke cusp forms}.
Taiwanese J. Math. \textbf{26} (2022), no. 3, 463--482.

\bibitem{LS24}
S.-C. Liu and J. Streipel, {\it The twisted second moment of L-functions associated to Hecke--Maass forms}.
Int. J. Number Theory \textbf{20} (2024), no. 3, 849--866.

\bibitem{LWY}
J. Liu, Y. Wang and Y. Ye, {\it  A proof of Selberg's orthogonality for automorphic $L$-functions}.
Manuscr. Math. \textbf{118} (2) (2005) 135--149.

\bibitem{Luo1}
W. Luo, {\it Zeros of Hecke L-functions associated with cusp forms}. Acta Arith. \textbf{71} (1995), no. 2, 139--158.

\bibitem{Montgomery}
H. L. Montgomery, {\it Zeros of L-functions}. Invent. Math. \textbf{8} (1969), 346--354.


\bibitem{Selberg}
A. Selberg, {\it On the remainder in the formula for $N(T)$, the number of zeros of $\zeta(s)$ in the strip $0<t<T$}.
Avh. Norske Vid.-Akad. Oslo I 1944 (1944), no. 1, 27 pp.

\bibitem{Selberg1}
A. Selberg, {\it Contributions to the theory of the Riemann zeta-function}.
Arch. Math. Naturvid. \textbf{48} (1946), no. 5, 89--155.

\bibitem{Selberg2}
A. Selberg, {\it Contributions to the theory of Dirichlet's $L$-functions}.
Skr. Norske Vid.-Akad. Oslo I 1946 (1946), no. 3, 62 pp.

\bibitem{SW}
Q. Sun and H. Wang, {\it On a level analog of Selberg's result on $S(t)$}. arXiv:2407.14867

\bibitem{SW2}
Q. Sun and H. Wang, {\it On an unconditional spectral analog of Selberg's result on $S(t)$}.
Ramanujan J. {\bf 67} (2025), no.~1, Paper No. 2, 31 pp.

\bibitem{SW24}
Q. Sun and H. Wang, {\it A zero-density estimate for L-functions associated
with $\rm GL(3)$ Hecke--Maass cusp forms.} arXiv:2412.02416

\bibitem{SW25}
Q. Sun and H. Wang, {\it On an unconditional  analog of Selberg's result}. arXiv:2502.01288


\bibitem{Titchmarsh1}
E. C. Titchmarsh, {\it On the Remainder in the Formula for $N(T)$,
the Number of Zeros of zeta(s) in the Strip $0<t<T$}. Proc. London Math. Soc. (2) \textbf{27} (1928), no. 6, 449--458.






\end{thebibliography}

\end{document}